\documentclass[12pt]{amsart}
 
\usepackage{graphicx, xcolor}
\usepackage{comment}
\usepackage{amsmath,amsthm,amsfonts,mathrsfs,MnSymbol,epsfig}
\usepackage[initials]{amsrefs}
\usepackage[all]{xy}
\usepackage{bm}
\usepackage{xcolor}

\usepackage{stackengine}
\stackMath

\providecommand{\CC}{{\mathbb{C}}}
\providecommand{\RR}{{\mathbb{R}}}

\providecommand{\mcD}{{\mathcal{D}}}
\providecommand{\mcE}{{\mathcal{E}}}
\providecommand{\mcF}{{\mathcal{F}}}

\providecommand{\mcO}{{\mathcal{O}}}
\providecommand{\mcS}{{\mathcal{S}}}

\providecommand{\msS}{{\mathscr{S}}}

\providecommand{\Lg}{{\mathfrak g}}

\providecommand{\Lg}{{\mathfrak g}}

\providecommand{\Symb}{{\mathscr{S}_H}}

\newcommand{\ft}{{s}}

\DeclareMathOperator{\Res}{Res}

\DeclareMathOperator{\supp}{supp}
\DeclareMathOperator{\sign}{sign}

\newtheorem{theorem}{Theorem}[section]
\newtheorem{lemma}[theorem]{Lemma}

\newtheorem{proposition}[theorem]{Proposition}

\theoremstyle{definition}
\newtheorem{definition}[theorem]{Definition}

\theoremstyle{remark}
\newtheorem{remark}[theorem]{Remark}
\newtheorem{example}[theorem]{Example}

\numberwithin{equation}{section}

\title[Traces for the Heisenberg Calculus]{A note on traces for the Heisenberg calculus}
\author{Alexander Gorokhovsky}
\address{University of Colorado, Boulder, Campus Box 395, Boulder, Colorado, 80309, USA}
\email{gorokhov@colorado.edu}
\author{Erik van Erp}
\address{Dartmouth College, 6188, Kemeny Hall, Hanover, New Hampshire, 03755, USA}
\email{jhamvanerp@gmail.com}

\begin{document}

\begin{abstract}
In \cite{GvE22}, we gave a local formula for the index of Heisenberg elliptic operators on contact manifolds. 
We constructed a cocycle in  periodic cyclic cohomology which, when paired with the Connes-Chern character of the principal Heisenberg symbol, calculates the index.
A crucial ingredient of our index formula was a new trace  on the   algebra of  Heisenberg pseudodifferential operators.
The construction of this trace was rather involved.
In the present paper we clarify the nature of this trace, and thus simplify the index formula of \cite{GvE22}.
\end{abstract}
\maketitle
\setcounter{tocdepth}{1}
\tableofcontents

\section{Introduction}

Let $\Lg=\RR^3$ be the Lie algebra of the Heisenberg group with basis $X, Y, Z$ and
\[ [X,Z]=[Y,Z]=0\qquad [X,Y]=Z\] 
Elements of $\Lg$ can de identified with right-invariant vector fields on the Heisenberg group $G=\RR^3$,
\[ X = \frac{\partial}{\partial x}-\frac{1}{2}y\frac{\partial}{\partial z}\qquad Y = \frac{\partial}{\partial y}+\frac{1}{2}x\frac{\partial}{\partial z}\qquad Z = \frac{\partial}{\partial z}\]
Here $(x, y, z)$ are the standard coordinates on $\RR^3$.
The Heisenberg calculus is a pseudodifferential calculus  that is built on the  idea that, for certain purposes, it is natural to treat $Z$ as a differential operator of degree 2.
For example, in the classical calculus, the leading term of the second order differential operator
\[ P = X^2+Y^2 + i\gamma Z\qquad  \gamma\in \CC\]
is the sublaplacian $X^2+Y^2$.
The principal symbol of $P$ is not invertible, and $P$ is not elliptic.
But in the Heisenberg calculus, the term $i\gamma Z$ has order 2 and is part of the leading term. 
It turns out that the  principal symbol of $P$ is invertible in the Heisenberg calculus iff  $\gamma$ is not an odd integer. 
In that case,   $P$ has a parametrix (an inverse modulo smoothing operators), and, while $P$ is not elliptic, it is still hypoelliptic.

The construction of the Heisenberg pseudodifferential calculus generalizes naturally to contact manifolds (see \cite{BG88}), which can be locally identified with the Heisenberg group.
On a compact contact manifold, an operator with invertible principal symbol in the Heisenberg calculus is a Fredholm operator.
In \cite{GvE22}, we gave a local formula for the index of Heisenberg elliptic operators on contact manifolds in terms of the principal symbol. 
Principal symbols of pseudodifferential operators in the Heisenberg calculus form a noncommutative algebra $\mathscr{S}_H$.
An invertible   symbol  determines an element in the  $K$-theory group  $K_1(\mathscr{S}_H)$.
The Connes-Chern character maps $K_1(\mathscr{S}_H)$ to  periodic cyclic homology $HP_1(\mathscr{S}_H)$.
In \cite{GvE22}, we constructed a cocycle in  periodic cyclic cohomology $HP^1(\mathscr{S}_H)$ which, when paired with the Connes-Chern character of the symbol,   calculates the index.

A crucial ingredient of our index formula was a  trace $\tau$ on the symbol algebra $\mathscr{S}_H$. The purpose of the present paper is to clarify the nature of this trace, and thus to simplify the index formula of \cite{GvE22}.
The construction of $\tau$ in \cite{GvE22} was rather involved.
We show here how $\tau$ can be understood with reference to a natural class of traces that can be defined  for a general homogeneous group.

A homogeneous group $G$ is a nilpotent Lie group with underlying manifold $\RR^m$ that has a graded Lie algebra.
Let $\Psi^0_H(G)$ be the algebra of translation invariant order zero operators in the Heisenberg calculus associated to  $G$ (as, for example, in \cite{CGGP92}).
An operator in $\Psi^0_H(G)$ is given by convolution (on the left) with a compactly supported  distribution $k\in \mcE'(G)$ that is regular (i.e., $k$ restricts to a smooth function on $G\minus \{0\}$) and that has an asymptotic expansion
\[ k\sim k_0+k_1+k_2+\cdots\]
The principal term $k_0$ is homogeneous of degree $-Q$, where $Q$ is the homogeneous dimension of $G$. 
In Section \ref{sec:traces_homogeneous} we show that evaluation of the principal term $k_0$ at a central element $z\in G$, $z\ne 0$   defines a trace,
\[ \tau_z:\Psi_H^0(G)\to \CC\qquad \tau_z(k):=k_0(z)\]
In Section \ref{sec:main} we briefly review our construction of the trace $\tau$ from \cite{GvE22}, which is defined if $G$ is the $2n+1$ dimensional Heisenberg group. 
The main result of this paper is Theorem \ref{thm:main}, which expresses $\tau$ as a linear combination of traces defined in Section \ref{sec:traces_homogeneous},
\begin{equation}\label{tau_new}
    \tau = \frac{(2\pi)^{n+1}}{2(n!) i^{n+1}}\left(\tau_{(0,+1)}-(-1)^n\tau_{(0,-1)}\right)
\end{equation} 
Section \ref{sec:fourier} introduces a technical lemma about the Fourier transform of a homogeneous distribution. 
We use this lemma in Section \ref{sec:proofs} to prove Theorem \ref{thm:main}. 

\begin{remark}
In \cite{GvE22}  we worked with traces on the Weyl algebra, rather than on the convolution algebra of the Heisenberg group.
The connection between the two involves a Fourier transform.
The constant in equation (\ref{tau_new}) depends on a choice of isomorphism between the convolution algebra and the algebra used in \cite{GvE22}.
The relevant conventions are spelled out in section \ref{conventions}.
\end{remark}

\subsection*{Acknowledgment}
We thank the anonymous referee for carefully reading the paper, and making useful suggestions and corrections. The first author is grateful for the hospitality and support of the Max-Planck Institute for Mathematics in Bonn where part of this work was carried out.

\section{Traces for homogeneous groups}\label{sec:traces_homogeneous}

\subsection{Homogeneous groups}
Let $G$ be a nilpotent Lie group with underlying manifold $\RR^n$, with graded Lie algebra $\Lg$. This means that $\Lg$ decomposes as a direct sum,
\[ \Lg=\Lg^{1} \oplus \cdots \oplus \Lg^{r}\]
where $[\Lg^j,\Lg^k]\subseteq  \Lg^{j+k}$
if $j+k\le r$, and $[\Lg^j,\Lg^k]$ = 0 if $j+k>r$.
Dilations $\delta_t:\Lg\to \Lg$  are defined by
\[ \delta_t(X) = t^{j}X\qquad \text{if}\;X\in \Lg^{j},\; t>0\]
The dilation $\delta_t$ is a Lie algebra automorphism.
We identify $G=\Lg$ via the exponential map.
Then $\delta_t$ is  a group automorphism.

If $\mu$ is a Haar measure on $G$,
then $\mu \circ \delta_t=t^Q\mu$, where  $Q$ is the homogeneous dimension of $G$,
\[ Q = \sum_{j=1}^r \;j \cdot\mathrm{dim}\Lg^{j}\]
We shall use the term ``homogeneous''  always in relation to the dilations $\delta_t$.
A function $f$ on $G$ is homogeneous of degree $k$ if $f(\delta_tx) = t^kf(x)$. A distribution $u\in \mcD'(G)$ is homogeneous of degree $k$ if
$\langle u, \varphi\circ\delta_t\rangle = t^{-Q-k} \langle u, \varphi\rangle$.

\begin{example}\label{ex:HGrp}
The Lie algebra of the Heisenberg group $G=\RR^{2n+1}$ is graded,
\[ \Lg=\Lg^1\oplus \Lg^2 \qquad \Lg^1=\RR^{2n},\; \Lg^2=\RR\]
The Lie bracket is
\[ [(x,t),(y,s)]=(0,\omega(x,y)) \qquad x,y\in \RR^{2n},\; t,s\in \RR\]
where $\omega(x,y) = \sum_{i=1}^n x_iy_{n+i}-x_{n+i}y_i$ is the standard symplectic form on $\RR^{2n}$.
The group operation is
\[ (x,t)(y,s)=(x+y, t+s+\frac{1}{2}\omega(x,y))\qquad x,y\in \RR^{2n},\; t,s\in \RR\]
The dilations are $\delta_\lambda(x,t) = (\lambda x, \lambda^2 t)$.
The homogeneous dimension is $Q=2n+2$.

\end{example}

\subsection{The algebra}
For a homogeneous group $G$ as above, let $\Psi^0_H(G)$ denote the algebra of right-invariant  pseudodifferential operators of order zero in the generalized Heisenberg calculus on $G$ (as defined, for example, in \citelist{\cite{CGGP92} \cite{Me83}}).
An operator in $\Psi^0_H(G)$ is given by convolution (on the left) with a compactly supported  distribution $k\in \mcE'(G)$ that is regular (i.e., smooth on $G\minus \{0\}$) and that has an asymptotic expansion
\[ k\sim k_0+k_1+k_2+\cdots\]
where the principal term $k_0$ is  homogeneous of degree $-Q$. 

The algebra $\msS_H$ of principal symbols is the quotient in the short exact sequence
\[ 0\to \Psi^{-1}_H(G) \to \Psi^0_H(G) \to \msS_H(G) \to 0\]
Elements of $\msS_H(G)$ can be represented by regular distributions on $G$ that are homogeneous of degree $-Q$.
Note that the Dirac delta distribution on $G$ is homogeneous of degree $-Q$.
The Dirac delta is the unit of the algebra $\msS_H(G)$.

Convolution of two homogeneous  distributions $f,g\in \msS_H(G)$ is defined as follows. 
Let $\phi \in C_c^\infty\left(G\right)$ be such that $\phi =1 $ in a neighborhood of $0$. 
Then $f * g$ is, by definition, the principal part (of homogeneous degree $-Q$) in the asymptotic expansion of $\phi f * \phi g$. The result does not depend on the choice of $\phi$.

\subsection{Traces}\label{sec:traces}
For a central element $z\in G$,
evaluation at $z$ defines a trace on the convolution algebra $C_c^\infty(G)$,
\begin{equation*}\label{eqn:conv}
(f\ast g)(z) = \int_G f(zx^{-1})g(x)dx = \int_G g(x)f(x^{-1}z)dx = (g\ast f)(z)
\end{equation*}
This generalizes in various ways.

\begin{lemma}\label{distr}
If $z$ is a central element of $G$, and $u\in \mathcal{E}'\left(G\right)$, $f \in C^{\infty}\left(G\right)$, then
\begin{equation*}
\left(u*f\right)\left(z\right)=\left(f*u\right)\left(z\right)
\end{equation*}    
\end{lemma}
\begin{proof}
This is immediate from the definitions,
\[ (u * f)(z) := \langle u, \tilde{f}^z\rangle\qquad \tilde{f}^z(y)=f(y^{-1}z)\]
and 
\[ (f * u)(z) := \langle u, \tilde{f}_z\rangle \qquad \tilde{f}_z(y)=f(zy^{-1})\]
\end{proof}

For a regular distribution $f$,
we can evaluate $f(z)$ as long as $z\ne 0$.

\begin{lemma}\label{thm:trace1}
If $z\ne 0$ is a central element in $G$, 
and $f, g$ are regular distributions with compact support on $G$, then
\[ (f * g)(z) = (g * f)(z)\]
\end{lemma}
\begin{proof}
Let $U$  be a neighborhood of $0 \in G$ such that $z \notin U^2$. Let $\phi \in C_c^\infty\left(G\right)$ be such that $\phi =1 $ in a neighborhood of $0$, and $\supp \phi \subset U$.    
Then,
\begin{equation*}
f * g=
\phi f * \phi g +\left(1-\phi\right)f * \phi g + 
\phi f * \left(1-\phi\right) g +
\left(1-\phi\right)f * \left(1-\phi\right) g
\end{equation*}
Since $(1-\phi)f, (1-\phi)g\in C_c^\infty(G)$ we have
\begin{equation*}
\left(1-\phi\right)f * \left(1-\phi\right) g \left(z\right) = \left(1-\phi\right) g * \left(1-\phi\right) f\left(z\right)
\end{equation*}
By Lemma \ref{distr},
\begin{equation*}
\left(1-\phi\right)f * \phi g \left(z\right) = \phi g * \left(1-\phi\right) f\left(z\right),\
\phi f * \left(1-\phi\right) g\left(z\right) =
\left(1-\phi\right)g * \phi f \left(z\right)
\end{equation*}
Finally, since $\supp \phi f * \phi g$ and    $\supp \phi g * \phi f$ are contained in $U^2$,  both of them vanish at $z$.

\end{proof}

\begin{proposition}
If $z\ne 0$ is a central element in $G$, then
\begin{equation*}
\tau_z:\msS_H(G)\to \RR \quad \tau_z\left(f\right) := f\left(z\right)
\end{equation*}
is a trace on $\mathscr{S}_H(G)$.
\end{proposition} 
\begin{proof}
By definition of the product in $\msS_H$(G),
\[ (f * g)(z) = \lim_{t\downarrow 0} t^Q (\phi f * \phi g) (\delta_t z)\]
which, by Lemma \ref{thm:trace1}, is equal to 
\[ (g * f)(z) = \lim_{t\downarrow 0} t^Q (\phi g * \phi f) (\delta_t z)\]

\end{proof}

\section{The main result}\label{sec:main}

In this section we specialize to the Heisenberg group.
Throughout this section, $G=\RR^{2n+1}$ will denote the Heisenberg group, with notations as in Example \ref{ex:HGrp}.

\subsection{Fourier transform}\label{conventions}

The Fourier transform is defined for Schwartz class functions $f \in \mcS(\RR^{2n+1})$ by 
\begin{equation*}
\hat{f}\left(y, \ft\right):= \int f\left(x, t\right) e^{-i\left(xy+t \ft\right)}dx dt
\end{equation*}
and is then extended to tempered distributions $\mcS'(\RR^{2n+1})$.

Given a compactly supported distribution $k\in \Psi_H^0(G)$, let $\hat{k}\in C^\infty(\RR^{2n+1})$ be its Fourier transform.
Define two functions $\sigma_\pm \in C^\infty(\RR^{2n})$,
\[ \sigma_+(x) := \lim_{\lambda\to +\infty} \hat{k}(\lambda x, \lambda^2)\qquad  \sigma_-(x) := \lim_{\lambda\to -\infty} \hat{k}(\lambda x, -\lambda^2)\qquad x\in \RR^{2n}\]
Alternatively, in the asymptotic expansion
\[ k\sim k_0+k_1+k_2+\cdots\]
the principal term $k_0$ is  homogeneous of degree $-Q$, and represents an element in $\Symb(G)$.
The Fourier transform $\hat{k}_0$ is a smooth function on $\RR^{2n+1}\minus \{0\}$ that is homogeneous of degree $0$ for the dilations, i.e.
\[ \hat{k}_0(\lambda x, \lambda^2t) = \hat{k}_0(x,t)\qquad x\in \RR^{2n}, t\in \RR, \lambda>0\]
Then
\[ \sigma_+(x) := \hat{k}_0(x,+1)\qquad \sigma_-(x) := \hat{k}_0(x,-1)\]
In particular, if $k\in \Psi_H^{-1}(G)$ then $k_0=0$ and $\sigma_{\pm}=0$.

If we wish to emphasize the dependence of $\sigma_\pm$ on $k$ we denote it as $\sigma^k_\pm$.
The map $k\mapsto \sigma_+^k$ is an algebra homomorphism from the convolution algebra $\Psi^0(G)$ to the Weyl algebra of $\RR^{2n}$,
\begin{equation}\label{sharp product}
    \sigma_+^{k\ast h}(v) = (\sigma^k_+\# \sigma^h_+)(v)  = \frac{1}{(2\pi)^{2n}} \iint e^{2i\,\omega(x,y)}\sigma^k_+(v+x)\sigma^h_+(v+y) dx\,dy
\end{equation} 
The map $k\to \sigma_-^k$ is an anti-homomorphism, with $\sigma_-^{k\ast h}=\sigma^h_-\# \sigma_-^k$.

\subsection{The old trace}

Our construction of the trace in \cite{GvE22} was strongly influenced by the  work of Epstein and Melrose in  \cite{EMxx}, \cite{EM98}. We summarize our construction here. For an  in-depth treatment, see \cite{GvE22}.

The smooth functions $\sigma_\pm$, defined above, have asymptotic expansions for large $\|x\|\to \infty$ of the form
\begin{equation}\label{eq:asymptotic_expansion}
\sigma_+(x) \sim \sum_{l=0}^\infty w_{2l}(x)\qquad \sigma_-(x) \sim \sum_{l=0}^\infty (-1)^l w_{2l}(x)    
\end{equation}
where $w_{2l}$ is smooth on $\RR^{2n}\minus 0$,
and homogeneous of degree $-2l$.

\begin{definition}\label{def:tau}
The trace $\tau:\Symb(G)\to \CC$  is 
\[ \tau(k) := \frac{1}{(2\pi)^n} \int_{\RR^{2n}}  \mathcal{R}(x) \,dx \qquad k\in \Symb(G)\]
where $\mathcal{R}$ is the function
\[ \mathcal{R} := \sigma_+ - (-1)^{n}\sigma_- - \sum_{l=0}^{n-1} \epsilon_lw_{2l}\qquad
\epsilon_l = 1-(-1)^{n+l} = 
\begin{cases}
0&l+n\;\text{even}\\
2&l+n\;\text{odd}
\end{cases}
\]
Note that $\mathcal{R}$ is an integrable smooth function defined on $\RR^{2n}\minus \{0\}$.
\end{definition}

For a proof that $\tau$ is a trace, see Theorem 5.8 and Proposition 5.10 in \cite{GvE22}.

\begin{remark}
If $\sigma_\pm$ are Schwartz class functions,
then $w_{2l}=0$ for all $l=0,1,2,\dots$, and simply
\[ \tau(k) = \frac{1}{(2\pi)^n} \int_{\RR^{2n}}  \sigma_+(x)\,dx - (-1)^{n} \frac{1}{(2\pi)^n} \int_{\RR^{2n}} \sigma_-(x) \,dx\]
In this case, the integral
\[\frac{1}{(2\pi)^n} \int_{\RR^{2n}}  \sigma_+(x)\,dx\]
is the usual trace of the trace class operator on $L^2(\RR^n)$
determined by the distribution $k_0$ in the Schr\"odinger representation of the Heisenberg group.

Likewise,
\[(-1)^n\frac{1}{(2\pi)^n} \int_{\RR^{2n}}  \sigma_-(x)\,dx\]
is the  trace of the trace class operator on $L^2(\RR^n)$ determined by $k$ in the dual of the Schr\"odinger representation.
(For details, see \cite{GvE22}, Section 5.3.)

\end{remark}

The main result of this paper is that the trace $\tau$ can be expressed as a linear combination of traces defined in Section \ref{sec:traces}.
\begin{theorem}\label{thm:main}
Let $G=\RR^{2n+1}$ be the Heisenberg group. 
Let 
\[\tau_+ := \tau_{(0,+1)}\qquad \tau_-:=\tau_{(0,-1)}\]
be the traces defined in Section \ref{sec:traces}, for the central elements $(0,\pm 1)\in G$.
Then
\[ \tau = \frac{(2\pi)^{n+1}}{2(n!) i^{n+1}}\left(\tau_{+}-(-1)^n\tau_{-}\right) \]
\end{theorem}

\section{A lemma about Fourier transforms}\label{sec:fourier}

For an integrable  function $g\in L^1(\RR^{2n+1})$, denote the Fourier transform by
\begin{equation*}
\mcF(g)(y, \ft):= \int g(x, t) e^{-i(xy+t \ft)}dx dt
\end{equation*}
with inverse Fourier transform
\begin{equation*}
\mcF^*(g)(y, \ft):= \frac{1}{(2\pi)^{2n+1}}\int g(x, t) e^{i(xy+t \ft)}dx dt
\end{equation*}
As usual, the  Fourier transform $\mcF(T)$ of a tempered distribution $T$ is defined by $\langle \mcF(T), \varphi\rangle = \langle T, \mcF(\varphi)\rangle $, $\varphi\in \mcS(\RR^{2n+1})$, and similarly for $\mcF^*(T)$.

In this section we prove the following lemma, which expresses the Fourier transform of a homogeneous distribution as a limit of the Fourier transforms of integrable functions.

\begin{lemma}\label{lem:beta_and_delta_to_zero}
Let $f$ be a smooth function on $\RR^{2n+1}\minus\{0\}$, homogeneous of degree $0$ with respect to the dilations, i.e.
\[ f(\lambda x, \lambda^2t) = f(x,t)\qquad x\in \RR^{2n}, t\in \RR, \lambda>0\]
Then
\[ \mcF^*(f)(0, \ft) = \lim_{\beta\downarrow 0}\lim_{\delta \downarrow 0}\mcF^*\left(f(x,t)e^{-\beta \|x\|^2 - \delta |t|}\right)\left(0, \ft\right)
\]
\end{lemma}
Note that $\mcF^*\left(f\right)$ is smooth  away from $0$, and so are $\mcF^*\left(fe^{-\beta \|x\|^2}\right)$ and $\mcF^*\left(fe^{-\beta \|x\|^2 - \delta |t|}\right)$ for $\delta, \beta >0$
The proof proceeds in two steps.

\begin{lemma} \label{lem:beta_to_zero}
Let $f$ be a smooth function on $\RR^{2n+1}\minus\{0\}$, homogeneous of order $0$ with respect to the dilations. Let $\beta >0$. For $\ft \ne 0$ 
\begin{equation*} \lim_{\delta \downarrow 0}\mcF^*\left(fe^{-\beta \|x\|^2 - \delta |t|}\right)\left(0, \ft\right)= 
\mcF^*\left(fe^{-\beta \|x\|^2  }\right)\left(0, \ft\right)\end{equation*}
\end{lemma}
\begin{proof}
Choose a smooth function $\chi \in C^\infty(\RR)$ with 
\[\chi(t) = \begin{cases}
    0&\text{if}\; |t|\le 1\\
    1&\text{if}\; |t|\ge 2
\end{cases}\]
Then $(1-\chi(t)) f(x,t)e^{-\beta \|x\|^2}$ is in $L^1\left(\RR^{2n+1}\right)$.
By the dominated convergence theorem,
\begin{equation}\label{first}
\lim_{\delta \downarrow 0}\mcF^*\left((1-\chi) f e^{-\beta \|x\|^2 - \delta |t|}\right)\left(0, \ft\right)=
\mcF^*\left((1-\chi) f e^{-\beta \|x\|^2}\right)\left(0, \ft\right)
\end{equation}
Assume $\beta>0$, $\delta>0$. Then $\chi(t)f(x,t)e^{-\beta\|x\|^2-\delta|t|}$ is a Schwartz class function. Therefore,
\begin{multline}\label{der}
\mcF^*\left(\chi f e^{-\beta\|x\|^2-\delta|t|}\right) =
\frac{1}{\left(-i\ft\right)^k} \mcF^*\left(\partial_t^k (\chi f e^{-\beta\|x\|^2-\delta|t|})\right)\\   
= \frac{1}{\left(-i\ft\right)^k}  \sum_{p=0}^k\binom{k}{p} \mcF^*\left(\partial_t^p (\chi f)\partial_t^{k-p} e^{-\beta\|x\|^2-\delta|t|}\right)  
\end{multline} 
Note that, while $e^{-\beta\|x\|^2-\delta|t|}$ is not smooth at $t\ne 0$,
the above formula makes sense because $\partial_t^p (\chi f)$ is zero for $t\in [-1,1]$.

Since differentiating in $t$ reduces the degree of homogeneity of $f$, for each $p=0,1,2,\dots$ there exists $C$ such that for all $(x,t)\in \RR^{2n+1}$:
\[ |\partial_t^p (\chi(x,t) f(x,t))| \le C |t|^{-p}  \]
Also
\[ |\partial_t^{k-p} e^{-\beta\|x\|^2- \delta |t|}| \le \delta^{k-p} e^{-\beta\|x\|^2-\delta|t|}\]
and so
\[ |\, \partial_t^p (\chi f)\partial_t^{k-p}e^{-\beta\|x\|^2-\delta|t|}\,|\le C \delta^{k-p}|t|^{-p}e^{-\beta\|x\|^2-\delta|t|}  \]
With $\beta>0$ fixed, we obtain (with different constant $C$),
\[ \|\,\partial_t^p (\chi f)\partial_t^{k-p}e^{-\beta\|x\|^2-\delta|t|}\,\|_{L^1}\le C \delta^{k-p} \int_{\|t\|>1} |t|^{-p}e^{-\delta |t|} dt\]
For small $\delta>0$,
\begin{equation*}
\int_{|t|\ge 1} |t|^{-p} e^{-\delta |t|}dt =\begin{cases}
  \mcO\left(\delta^{-1}\right) &\text{ if } p=0\\
    \mcO\left(\log \delta\right) &\text{ if } p=1\\
    \mcO\left(1\right) &\text{ if } p\ge 2
\end{cases}
\end{equation*}
and so
\begin{equation*}
\|\,\partial_t^p (\chi f)\partial_t^{k-p}e^{-\beta\|x\|^2-\delta|t|}\,\|_{L^1} =\begin{cases}
  \mcO\left(\delta^{k-1}\right) &\text{ if } p=0\\
    \mcO\left(\delta^{k-1}\log \delta\right) &\text{ if } p=1\\
    \mcO\left(\delta^{k-p}\right) &\text{ if } p\ge 2
\end{cases}
\end{equation*}
We now assume that $k\ge 2$. 
Then all these $L^1$-norms converge to $0$ if $\delta\downarrow 0$, except when $p=k$. 
Since $\|\mcF^*(g)\|_\infty \le \|g\|_{L^1}$, equation (\ref{der}) implies
\[
\lim_{\delta\downarrow 0} \mcF^*\left(\chi f e^{-\beta\|x\|^2-\delta|t|}\right)(0,s) 
=  \lim_{\delta\downarrow 0}   \frac{1}{\left(-i\ft\right)^k} \mcF^*\left(\partial_t^k (\chi f)\,e^{-\beta\|x\|^2-\delta|t|}\right)  (0,s)
\]
Since $\partial_t^k(\chi f)e^{-\beta \|x\|^2}$ is in $L^1$,  dominated convergence  gives,
\begin{multline*}
\lim_{\delta \downarrow 0} \mcF^*\left(\partial_t^k (\chi f)e^{-\beta \|x\|^2 - \delta |t|}\right)\left(0, \ft\right)=
\mcF^*\left(\partial_t^k (\chi f)e^{-\beta \|x\|^2 }\right)\left(0, \ft\right)\\
= \mcF^*\left(\partial_t^k (\chi f e^{-\beta \|x\|^2 })\right)\left(0, \ft\right) = (-is)^k \mcF^*\left(\chi f e^{-\beta \|x\|^2 }\right)\left(0, \ft\right)
\end{multline*}
In other words,
\[
\lim_{\delta\downarrow 0} \mcF^*\left(\chi f e^{-\beta\|x\|^2-\delta|t|}\right)(0,s) 
= \mcF^*\left(\chi f e^{-\beta \|x\|^2 }\right)\left(0, \ft\right)
\]
Combined with \eqref{first} this completes the proof.

\end{proof}
\begin{lemma}  
Let $f$ be a smooth function on $\RR^{2n+1}\minus\{0\}$, homogeneous of degree $0$ with respect to the dilations. For $\ft \ne 0$,
\begin{equation*} \lim_{\beta \downarrow 0}\mcF^*\left(fe^{-\beta \|x\|^2 }\right)\left(0, \ft\right)= 
\mcF^*\left(f\right)\left(0, \ft\right)\end{equation*}
\end{lemma}
\begin{proof}
Let $\varphi \in C^\infty_c\left(\RR^{2n+1}\right)$ be a compactly supported smooth function such that $\varphi(x,t)=1$ in a neighborhood of $0$.
Then $\varphi f$ is integrable and so
\[
\lim_{\beta \downarrow 0}\mcF\left(\varphi fe^{-\beta \|x\|^2 }\right)\left(0, \ft\right)= 
\mcF\left(\varphi f \right)\left(0, \ft\right)
\]
We also have
\[
\mcF^*\left(\left(1-\varphi\right) fe^{-\beta \|x\|^2 }\right) = \frac{1}{\left(-i \ft\right)^k} \mcF^*\left(\partial_t^k\left(\left(1-\varphi\right) f\right)e^{-\beta \|x\|^2 }\right) 
\]
$\partial_t^k\left(\left(1-\varphi\right) f\right)$ is integrable  if  $k$ is sufficiently large, and hence
\begin{align*}
\lim_{\beta \downarrow 0}\mcF^*\left(\partial_t^k\left(\left(1-\varphi\right) f\right)e^{-\beta \|x\|^2 }\right) \left(0, \ft\right)
&=
 \mcF^*\left(\partial_t^k\left(\left(1-\varphi\right) f\right)\right) \left(0, \ft\right) \\
 &=\left(-i \ft\right)^k \mcF^*\left(\left(1-\varphi\right) f\right) \left(0, \ft\right)
\end{align*}
The statement follows.

\end{proof}

\section{Proof of the main result}\label{sec:proofs}

Throughout this section, $G=\RR^{2n+1}$ is the Heisenberg group, with notations as in Example \ref{ex:HGrp}.

The trace $\tau_z(k)$  is, by definition, evaluation of the principal part $k_0$ of $k$ at a central element $z\in G$, $z\ne 0$. 
The definition of the trace $\tau(k)$ is expressed in terms of the principal Heisenberg symbol $(\sigma_+,\sigma_-)$, which involves the Fourier transform of $k$.
In order to relate  $\tau$ and $\tau_z$, we use Lemma \ref{lem:beta_and_delta_to_zero} to express $\tau_z(k)$ in terms of  $(\sigma_+,\sigma_-)$. 

From there, the proof of  Theorem \ref{thm:main} consists of a series of calculations.
These calculations result in Proposition \ref{prop:t_z_formula}, which expresses $\tau_z$ as a linear combination of $\tau$ and a well-known residue trace (see (\ref{eqn:res}) below).
Theorem \ref{thm:main} follows immediately from Proposition \ref{prop:t_z_formula}.

For readability, we have formalized  steps in these calculations as lemmas.

\vskip 6pt

\begin{lemma}\label{lem:tau_z_as_beta_limit}
Let  $k\in \Psi^0_H(G)$, and let  $z=(0,s)\in G$  be a central element with $s\ne 0$.
Then
\[ \tau_z(k) = \frac{n!}{(2\pi)^{2n+1}} \;\lim_{\beta\downarrow 0}
\int_{\RR^{2n}}\left( \frac{ \sigma_+\left(x\right) }{ \left(\beta \|x\|^2-i \ft\right)^{n+1}} 
+ \frac{\sigma_-\left(x\right) }{ \left(\beta \|x\|^2+i \ft\right)^{n+1}}\right)dx\]
\end{lemma}

\begin{proof}
Let $\delta >0$. Denote $f=\hat{k}_0$.
For $t>0$ we have $f\left(x, t\right) = \sigma_+\left(x/\sqrt{t}\right)$, for $t <0$: $f\left(x, t\right) =\sigma_-\left(x/\sqrt{-t}\right)$.
Change of variables gives 
\begin{multline*}
(2\pi)^{2n+1}\mcF^*\left(fe^{-\beta \|x\|^2 - \delta |t|}\right)\left(0, \ft\right) =
\int  fe^{-\beta \|x\|^2 - \delta |t| + i \ft t} dt dx=\\
 \int_{\RR^{2n}} \int_0^\infty \sigma_+\left(x/\sqrt{t}\right)e^{-\beta \|x\|^2 - \delta t + i \ft t} \,dtdx
+ \int_{\RR^{2n}}\int_{-\infty}^0 \sigma_-\left(x/\sqrt{-t}\right) e^{-\beta \|x\|^2 + \delta t + i \ft t}  \,dtdx=\\
\int_{\RR^{2n}} \sigma_+\left(x\right)\int_0^\infty t^n e^{-(\beta \|x\|^2 + \delta  - i \ft) t} \,dt dx
+ \int_{\RR^{2n}} \sigma_-\left(x\right) \int^{\infty}_0 t^n e^{-(\beta \|x\|^2 + \delta   + i \ft) t}  \,dt dx
\end{multline*}
Using the Gamma integral,
\[ \int_0^\infty t^n e^{-ct}\,dt = \frac{1}{c^{n+1}}\int_0^\infty u^ne^{-u}\,du = \frac{\Gamma(n+1)}{c^{n+1}} = \frac{n!}{c^{n+1}}\]
we obtain
\[\mcF^*\left(fe^{-\beta \|x\|^2 - \delta |t|}\right)\left(0, \ft\right)
=\frac{n!}{(2\pi)^{2n+1}} \int_{\RR^{2n}}\left(\frac{ \sigma_+\left(x\right)}{ \left(\beta \|x\|^2+\delta-i \ft\right)^{n+1}} + \frac{  \sigma_-\left(x\right)}{ \left(\beta \|x\|^2+\delta+i \ft\right)^{n+1}}\right)\,dx
\]
The statement now follows from Lemma \ref{lem:beta_and_delta_to_zero}.
 
\end{proof}

With $k\in \Psi^0_H(G)$ as above, and the asymptotic expansions of $\sigma_\pm$ as in (\ref{eq:asymptotic_expansion}), define
\[
\Tilde{\sigma}_+:= \begin{cases}
\sigma_+ - \sum_{l=0}^n w_{2l} &\text{ if } \|x\|\ge 1\\
\sigma_+ - \sum_{l=0}^{n-1} w_{2l} &\text{ if } \|x\|< 1
\end{cases}
\]
and
\[
\Tilde{\sigma}_-:= \begin{cases}
\sigma_- - \sum_{l=0}^n (-1)^l w_{2l}  &\text{ if } \|x\|\ge 1\\
\sigma_- - \sum_{l=0}^{n-1} (-1)^l w_{2l}  &\text{ if } \|x\|< 1
\end{cases}
\]
Note that $\Tilde{\sigma}_\pm$ are integrable functions on $\RR^{2n}\minus \{0\}$.

\begin{lemma}
\[
 \frac{n!}{(2\pi)^{2n+1}}\; \lim_{\beta \downarrow 0} 
 \int_{\RR^{2n}}\left( \frac{ \Tilde{\sigma}_+\left(x\right) }{ \left(\beta \|x\|^2-i \ft\right)^{n+1}} + \frac{\Tilde{\sigma}_-\left(x\right) }{ \left(\beta \|x\|^2+i \ft\right)^{n+1}}\right)dx=
 \frac{n!}{(2\pi)^{n+1}(-is)^{n+1}}\tau(k)
\]
\end{lemma}
\begin{proof}
Since $\Tilde{\sigma}_\pm$ are integrable, the left-hand side is simply
\[\frac{n!}{(-is)^{n+1}(2\pi)^{2n+1}}\int_{\RR^{2n}} (\Tilde{\sigma}_+(x) +(-1)^{n+1}  \Tilde{\sigma}_-(x))\,dx\]
This is equal to the right-hand side, since 
\[\mathcal{R} = \Tilde{\sigma}_+ +(-1)^{n+1} \Tilde{\sigma}_-\]
where $\mathcal{R}$ is as in Definition \ref{def:tau}.

\end{proof}
We get
\begin{equation}\label{eq:tau_s_one}
\tau_z(k) =  \frac{n!}{(2\pi)^{n+1}(-is)^{n+1}} \tau(k) + \sum_{l=0}^n I_l    
\end{equation} 
where, for $l=0,1,\dots, n-1$,
\[ I_l = \frac{n!}{(2\pi)^{2n+1}} \;\lim_{\beta\downarrow 0}
\int_{\RR^{2n}}\left( \frac{ w_{2l}(x) }{ \left(\beta \|x\|^2-i \ft\right)^{n+1}} 
+ \frac{(-1)^lw_{2l}(x) }{ \left(\beta \|x\|^2+i \ft\right)^{n+1}}\right)dx\]
The formula for $I_n$ is the same (with $l=n$), except that the domain of integration is $\|x\|\ge 1$.

\begin{lemma}\label{lem:I_l}
For $l=0,1,\dots n-1$ we have $I_l=0$. 
\end{lemma}
\begin{proof}
Introduce spherical coordinates 
\[ r:=\|x\| \ge 0\qquad \theta:=\frac{x}{r} \in S^{2n-1}\qquad x\in \RR^{2n}\]
and let
\[ w_{2l}(x) = a_{2l}(x)r^{-2l} \qquad a_{2l} \in C^\infty(S^{2n-1})\]
We have
\begin{multline*}
 \int_{\RR^{2n}}\left( \frac{ a_{2l}(\theta)r^{-2l} }{ \left(\beta r^2-i \ft\right)^{n+1}} +  \frac{(-1)^la_{2l}(\theta)r^{-2l}  }{ \left(\beta r^2+i \ft\right)^{n+1}}\right)dx =\\
\int_{S^{2n-1}} a_{2l}(\theta) d \theta \int_0^{\infty} \left( \frac{r^{-2l} }{ \left(\beta r^2-i \ft\right)^{n+1}} +  \frac{(-1)^l r^{-2l}  }{ \left(\beta r^2+i \ft\right)^{n+1}}\right)r^{2n-1} dr   
\end{multline*}
and
\begin{multline*}
\int_0^{\infty} \left( \frac{r^{-2l} }{ \left(\beta r^2-i \ft\right)^{n+1}} +  \frac{(-1)^l r^{-2l}  }{ \left(\beta r^2+i \ft\right)^{n+1}}\right)r^{2n-1} dr   =\\
\frac{1}{2}\int_0^{\infty} \left( \frac{z^{-l+n-1} }{ \left(\beta z-i \ft\right)^{n+1}} +  \frac{(-1)^l z^{-l+n-1}  }{ \left(\beta z+i \ft\right)^{n+1}}\right)dz =
\frac{1}{2}\int_{-\infty}^{\infty}  \frac{z^{-l+n-1} }{ \left(\beta z-i \ft\right)^{n+1}} dz
\end{multline*}
We used a change of variables $z\mapsto -z$ for the second summand.

Contour integration over an interval $[-R, R]$ and a semicircle of radius $R$ in the upper half-plane if $\ft <0$, or the lower half-plane if $\ft >0$, shows that 
\[
\int_{-\infty}^{\infty}  \frac{z^{-l+n-1} }{ \left(\beta z-i \ft\right)^{n+1}} dz=0
\]
\end{proof}

With  $a_{2l}\in C^\infty(S^{2n-1})$ as above, define
\begin{equation}\label{eqn:res}
    \Res(k) = -\frac{1}{2 (2\pi)^n}\int_{S^{2n-1}} a_{2n}(\theta) d \theta
\end{equation} 
This is a trace on $\Symb(G)$ (see \cite{GvE22}).

\begin{lemma}\label{lem:I_n}
\[ I_n = \sign(s) \frac{n!}{(2\pi)^{n+1}(-is)^{n+1}}\,\pi i\Res(k)\]
\end{lemma}
\begin{proof}

To calculate $I_n$, we use spherical coordinates, as above. We have
\begin{multline*}
 \int_{r \ge 1} \left( \frac{ a_{2n}(\theta)r^{-2n}  }{ \left(\beta r^2-i \ft\right)^{n+1}} +  \frac{(-1)^na_{2n}(\theta)r^{-2n}   }{ \left(\beta r^2+i \ft\right)^{n+1}}\right)r^{2n-1}dr d\theta =\\
\int_{S^{2n-1}} a_{2n}(\theta) d \theta \int_1^{\infty} \left( \frac{1}{ \left(\beta r^2-i \ft\right)^{n+1}} +  \frac{(-1)^l    }{ \left(\beta r^2+i \ft\right)^{n+1}}\right)r^{-1} dr   
\end{multline*}
Thus
\[ I_n = -\frac{2 (n!)}{(2\pi)^{n+1}}  \Res(k) \cdot \lim_{\beta\downarrow 0} \int_1^{\infty} \left( \frac{1}{ \left(\beta r^2-i \ft\right)^{n+1}} +  \frac{(-1)^l    }{ \left(\beta r^2+i \ft\right)^{n+1}}\right)r^{-1} dr  \]
As before
\begin{multline}
  \int_1^{\infty} \left( \frac{1}{ \left(\beta r^2-i \ft\right)^{n+1}} +  \frac{(-1)^l    }{ \left(\beta r^2+i \ft\right)^{n+1}}\right)r^{-1} dr   =\\
\frac{1}{2}\int_{z\in \RR, |z|\ge 1}  \frac{z^{-1} }{ \left(\beta z-i \ft\right)^{n+1}} dz =  \frac{1}{2}\int_{z\in \RR, |z|\ge \beta}  \frac{z^{-1} }{ \left(z-i \ft\right)^{n+1}}  dz. 
\end{multline}
A standard application of the residue theorem shows that 
\[
\lim_{\beta \downarrow 0}\int_{z\in \RR, |z|\ge \beta}  \frac{z^{-1} }{ \left(z+i \ft\right)^{n+1}}  dz= \begin{cases}
-\frac{ \pi i}{(-i\ft)^{n+1}} &\text{ if } \ft >0 \\
\frac{ \pi i}{(-i\ft)^{n+1}} &\text{ if } \ft <0
\end{cases}
\]
This proves the Lemma.

\end{proof}

\begin{proposition}\label{prop:t_z_formula}
Let $z =(0, s)\in G$ be a central element with $s\ne 0$. Then
 \begin{equation*}
 \frac{(2\pi)^{n+1}(-i\ft)^{n+1}}{n!} \,\tau_z= 
 \begin{cases}
 \tau +\pi i \Res &\text{ if } \ft >0 \\
 \tau -\pi i \Res &\text{ if } \ft <0
\end{cases} 
 \end{equation*}
\end{proposition}
\begin{proof}
Combine (\ref{eq:tau_s_one}) with Lemmas \ref{lem:I_l} and \ref{lem:I_n}.
\end{proof}


\end{document}